\documentclass{my_gOPT2e}

\begin{document}

\markboth{A.B. Malinowska and N. Martins}{A.B. Malinowska and N. Martins}

\articletype{}

\title{Generalized transversality conditions\\
 for the Hahn quantum variational calculus}

\author{Agnieszka B. Malinowska$^{\rm a}$$^{\ast}$\thanks{$^\ast$Corresponding author. Email: abmalinowska@ua.pt
\vspace{6pt}} and Nat\'{a}lia Martins$^{\rm b}$\\\vspace{6pt}  $^{\rm a}${\em{Faculty of Computer Science, Bia{\l}ystok University of Technology, 15-351 Bia\l ystok, Poland}}; $^{\rm b}${\em{Department of Mathematics, University of Aveiro, 3810-193 Aveiro, Portugal } }}

\date{}

\maketitle

\begin{abstract}
We prove optimality conditions for generalized quantum variational
problems with a Lagrangian depending on the free end-points. Problems
of calculus of variations of this type cannot be solved using the
classical theory.\bigskip

\begin{keywords} Hahn's difference operator;
Jackson--Norl$\ddot{u}$nd's integral; quantum calculus; calculus of
variations; Euler--Lagrange equation; generalized natural boundary
conditions.
\end{keywords}
\begin{classcode}39A13, 39A70, 49J05, 49K05, 49K15.\end{classcode}\bigskip

\end{abstract}


\section{Introduction}

The (classical) calculus of variations is an old branch of mathematics that has many applications
in physics, geometry, engineering, dynamics, control theory, and economics.
The basic problem of calculus of variations can be formulated as follows: among all
differentiable functions $y:[a,b]\rightarrow\mathbb{R}$  such that $y(a) = \alpha$ and $y(b) = \beta$,
where $\alpha, \beta$ are fixed real numbers,
find the ones that minimize (or maximize) the functional
$$\mathcal{L}[y]=\int_a^b L(t,y(t),y^\prime(t)) dt.$$
It can be proved that the candidates to be minimizers or maximizers
to this basic problem must satisfy the differential equation
$$\frac{d}{dt}\partial_3 L(t,y(t),y^\prime(t))= \partial_2 L(t,y(t),y^\prime(t))$$
called the Euler--Lagrange equation
(where $\partial_i L$ denotes the partial derivative of $L$ with respect to its $i$th argument).
If the boundary condition $y(a) = \alpha$ is not present in the problem, then to find the
candidates for extremizers we have to
add another necessary condition:
$\partial_3 L(a,y(a),y^\prime(a))=0$; if $y(b) = \beta$ is not present, then $\partial_3 L(b,y(b),y^\prime(b))=0$.
These two conditions are usually called natural boundary conditions.

However, many important physical phenomena are described by nondifferentiable functions.
Several different approaches to deal with nondifferentiable functions
are proposed in the literature of variational calculus. In this paper we follow
the new \emph{Hahn quantum variational approach} \cite{MalinowskaTorres,BritoMT}.

The Hahn difference operator, $D_{q,\omega}$, was introduced in 1949 by Hahn \cite{Hahn}
and is defined by
\[
D_{q,\omega}\left[  f\right]  \left(  t\right)  :=\frac{f\left(  qt+\omega
\right)  -f\left(  t\right)  }{\left(  q-1\right)  t+\omega},  \quad t\neq\omega_0
\]
where $q\in ]0,1[$ and $\omega>0$ are real fixed numbers, $\omega_{0}:=\displaystyle\frac{\omega}{1-q}$,  and
$f$ is a real function defined on an interval $I$ containing $\omega_0$.

The Hahn difference operator has been applied successfully in the construction
of families of ortogonal polynomials as well as in approximation problems
\cite{alvares,odzi,Petronilho}. However, during 60 years, the construction of the proper inverse
of Hahn's difference operator remained as an open question.
The problem was solved in 2009 by Aldwoah \cite{1}
(see also \cite{Aldwoah:IJMS}).

The Hahn quantum variational calculus was started in 2010 with the
work \cite{MalinowskaTorres}. In that paper, among other results,
the authors formulated the basic and isoperimetric problems of the calculus of
variations with the Hahn derivative and obtained the respective Euler--Lagrange equations.
The Euler--Lagrange equation for quantum
variational problems involving Hahn's derivatives of higher-order
was obtained in \cite{BritoMT}. The purpose of this paper is to
present optimality conditions for generalized quantum variational problems. The work is
motivated by an economic problem which is explained in
\cite{Zinober}. Briefly the economic nature of the problem lies in
the effect of permitting the royalty in the profit maximizing firm
problem. This more general form leads naturally to new kind of problems in
calculus of variations and can be formulated in the
following way: what are the necessary optimality conditions for the
problem of the calculus of variations with a free end-point $y(b)$
but whose Lagrangian depends explicitly on $y(b)$? Terminal
conditions, which are also known as the transversality conditions
are important in economic policy models (for a deeper discussion we
refer the reader to \cite{Sen}): the optimal control or decision
rules are not unique without these boundary conditions. Our object
here is to state the natural boundary conditions for a dynamic
adjustment model. Assuming that due to some constraints of
economical nature the dynamic does not depend on the usual
derivative or the forward difference operator, but on the Hahn
quantum difference operator $D_{q,\omega}$, we present the
Euler--Lagrange equation and the natural boundary conditions for this
model. Our assumption is connected with a moot question: what kind
of ``time'' (continuous or discrete) should be used in the
construction of dynamic models in economics? Although individual
economic decisions are generally made at discrete time intervals, it
is difficult to believe that they are perfectly synchronized as
postulated by discrete models. The usual assumption that the
economic activity takes place continuously, is a convenient
abstraction in many applications. In others, such as the ones
studied in financial market equilibrium, the assumption of
continuous trading corresponds closely to reality. \\
One of the approaches proposed in the literature to deal with the question of time mentioned above, is
the time scale approach, which typically deals with delta-differentiable (or nabla-differentiable)
functions \cite{A:U:08,Atici,Z:N:D,Rui,R:A:D,M:T,infHorizon,nat:del,N:T,N:T:2}.
The origins of this idea dates back to the late 1980's when
S.~Hilger introduced this notion in his Ph.D. thesis (directed by B.~Aulbach)
and showed how to unify continuous time and discrete time dynamical
systems \cite{Hilger}. However, the Hahn quantum calculus is not covered
by the Hilger time scale theory. This is well explained
in the 2009 Ph.D. thesis of Aldwoah \cite{1} (see also \cite{Aldwoah:IJMS}).
Here we just note the following: the main advantage of the Hahn quantum variational
calculus is that we are able
to deal with nondifferentiable functions,
even discontinuous functions. Variational problems in the time scale setting
are formulated for functions that are delta-differentiable (or nabla-differentiable).
It is well known that delta-differentiable functions are necessarily continuous.
This is not the case in the Hahn quantum calculus: see Example~\ref{non:dif} (also Subsection~3.3 in \cite{BritoMT}),
where a discontinuous function is $q,\omega$-differentiable
in all the real interval $[-1,1]$.

The paper is organized as follows. In Section~\ref{sec:prelim} we
summarize all the necessary definitions and properties of the Hahn
difference operator and the associated $q,\omega$-integral. In
Section~\ref{sec:mr} we formulate the more general problem of the
calculus of variations with a Lagrangian that may also depend on the
unspecified end-points $y(a)$ and  $y(b)$. Then, we  prove our main
results: the Euler--Lagrange equation (Theorem~\ref{E-L}),
natural boundary conditions (Theorem~\ref{Theorem natural
boundary conditions}), necessary optimality conditions
for isoperimetric problems (Theorem~\ref{normalcase} and Theorem~\ref{thm:abn}),
and a sufficient optimality condition for variational problems (Theorem~\ref{suff}). Section~\ref{sec:Ex}
provides concrete examples of application of our results. We end with Section~\ref{sec:con} of conclusions and future perspectives.


\section{Preliminaries}
\label{sec:prelim}

Let $q\in]0,1[$ and $\omega\geq0 \footnote{Although Hahn and Aldwoah considered only $\omega >0$, the theory works well if we consider also $\omega=0.$}$. Define
\[
\omega_{0}:=\frac{\omega}{1-q}
\]
and let  $I$ be a real interval containing $\omega_{0}$.
For a function $f$ defined on $I$,
the \emph{Hahn difference operator} of $f$ is given by
\begin{equation*}
\displaystyle D_{q,\omega} [f](t):=
\begin{cases} \displaystyle
\frac{f(qt+\omega)-f(t)}{(q-1)t+\omega} & \text{ if } t\neq \omega_0\\
&\\
\displaystyle f'( \omega_0) & \text{ if } t = \omega_0
\end{cases}
\end{equation*}
provided that $f$ is differentiable at $\omega_0$
(where  $f'$  denotes the Fr\'{e}chet derivative of $f$).
 $D_{q,\omega}\left[f\right]$ is called
the $q,\omega$\emph{-derivative of} $f$, and $f$ is said to be $q,\omega$\emph{-differentiable on} $I$
if $D_{q,\omega}\left[f\right]\left(\omega_{0}\right)$ exists.

\begin{remark}
Note that
when $q\rightarrow 1$ we obtain the forward $h$-difference operator
\[
\Delta_{h}\left[f\right]\left(t\right)
:=\frac{f\left(t+h\right) - f\left(t\right)}{h},
\]
and when $\omega= 0$ we obtain the Jackson $q$-difference operator
\begin{equation*}
\displaystyle D_{q,0} [f](t):=
\begin{cases} \displaystyle
\frac{f(qt)-f(t)}{(q-1)t} & \text{ if } t\neq 0\\
&\\
\displaystyle f'(0) & \text{ if } t = 0
\end{cases}
\end{equation*}
provided
$f^{\prime}\left(0\right)$ exists.
Hence, we can state that the $D_{q,\omega}$ operator generalizes  the forward
$h$-difference and the Jackson $q$-difference operators \cite{Ernst,Kac}.

Notice also that,
under appropriate conditions,
\[
\lim_{q\rightarrow1}D_{q,0}\left[f\right]\left(t\right)
=f^{\prime}\left(t\right).
\]
\end{remark}

\begin{example} (\cite{MalinowskaTorres,BritoMT}) Let $q=\omega=1/2$. In this case $\omega_0=1$.
It is easy to see that $f : [-1,1] \rightarrow \mathbb{R}$ given by
\begin{equation*}
f(t) =
\begin{cases}
-t & \text{ if } t \in ]-1,0[\cup ]0,1]\\
0 & \text{ if } t=-1\\
1 & \text{ if } t=0
\end{cases}
\end{equation*}
is not a continuous function but
is $q,\omega$-differentiable in $[-1,1]$ with
\begin{equation*}
D_{q,\omega} [f](t) =
\begin{cases}
-1 & \text{ if } t \in ]-1,0[\cup ]0,1]\\
1 & \text{ if } t=-1\\
-3 & \text{ if }  t=0.
\end{cases}
\end{equation*}
\end{example}

\begin{example} (\cite{MalinowskaTorres})
\label{non:dif}
Let $q\in]0,1[$, $\omega=0$, and
\begin{equation*}
f(t)=
\begin{cases}
t^2 & \text{ if } t\in\mathbb{Q}\\
-t^2 & \text{ if } t\in\mathbb{R}\setminus\mathbb{Q}.
\end{cases}
\end{equation*}
Note that $f$ is only Fr\'{e}chet differentiable in zero, but since
 $\omega_0=0$, $f$ is $q,\omega$-differentiable on the entire real line.
\end{example}

The Hahn difference operator has the following properties:

\begin{theorem}(\cite{1,Aldwoah:IJMS})
If $f,g:I\rightarrow\mathbb{R}$ are $q,\omega$-differentiable  and $t\in I$, then:
\begin{enumerate}
\item  $D_{q,\omega}[f](t) \equiv 0$ on $I$ if and only if $f$ is constant;

\item $D_{q,\omega}\left[  f+g\right]  \left(  t\right)  =D_{q,\omega}\left[
f\right]  \left(  t\right)  +D_{q,\omega}\left[  g\right]  \left(  t\right)$;

\item $D_{q,\omega}\left[  fg\right]  \left(  t\right)  =D_{q,\omega}\left[
f\right]  \left(  t\right)  g\left(  t\right)  +f\left(  qt+\omega\right)
D_{q,\omega}\left[  g\right]  \left(  t\right)$;

\item $\displaystyle D_{q,\omega}\left[  \frac{f}{g}\right]  \left(  t\right)
=\frac{D_{q,\omega}\left[  f\right]  \left(  t\right)  g\left(  t\right)
-f\left(  t\right)  D_{q,\omega}\left[  g\right]  \left(  t\right)  }{g\left(
t\right)  g\left(  qt+\omega\right)  }$ if $g\left(
t\right)  g\left(  qt+\omega\right)  \neq0$;

\item $f\left(  qt+\omega\right)  =f\left(  t\right)  +\left(  t\left(
q-1\right)  +\omega\right)  D_{q,\omega}\left[  f\right]  \left(  t\right)  $.
\end{enumerate}
\end{theorem}

\begin{proposition}(\cite{1})
Let $a,b\in\mathbb{R}$. We have
$$\label{tpower}
D_{q,\omega}(at+b)^n=a\sum_{k=0}^{n-1}(a(qt+\omega)+b)^k(at+b)^{n-k-1},
$$
for $n\in\mathbb{N}$ and $t \neq \omega_0$.
\end{proposition}

Let $\sigma\left(t\right)  =qt+\omega$, for all $t\in I$.
Note that $\sigma$ is a contraction, $\sigma(I)\subseteq I$,
$\sigma\left(t\right)<t$ for $t>\omega_{0}$,
$\sigma\left(  t\right)>t$ for $t<\omega_{0}$,
and $\sigma\left(  \omega_{0}\right)  =\omega_{0}$.

We use the following standard notation of $q$-calculus:
for $k\in\mathbb{N}_{0}:=\mathbb{N}\cup\left\{  0\right\}$,
$\displaystyle\left[  k\right]  _{q}:=\frac{1-q^{k}}{1-q}$.

\begin{lemma}(\cite{1})
Let $k\in\mathbb{N}$ and $t\in I$. Then,
\begin{enumerate}
\item
$\sigma^{k}\left(  t\right)  =\underset{k\text{-times}}{\underbrace{\sigma
\circ\sigma\circ \cdots \circ\sigma}}\left(  t\right)  =q^{k}t+\omega\left[
k\right]_{q}$;

\item
$\displaystyle \left(  \sigma^{k}\left(  t\right)  \right)  ^{-1}=\sigma^{-k}\left(
t\right)  =\frac{t-\omega\left[  k\right]  _{q}}{q^{k}}.
$
\end{enumerate}
\end{lemma}

Following \cite{1,Aldwoah:IJMS} we define the notion of
$q,\omega$-\emph{integral} (also known as the \emph{Jackson--N\"{o}rlund integral}) as follows:

\begin{definition}
Let  $a,b\in I$ and $a<b$. For $f:I\rightarrow\mathbb{R}$
the $q,\omega$-\emph{integral of} $f$ from $a$ to $b$ is given by
\[
\int_{a}^{b}f\left(  t\right)  d_{q,\omega}t:=\int_{\omega_{0}}^{b}f\left(
t\right)  d_{q,\omega}t-\int_{\omega_{0}}^{a}f\left(  t\right)  d_{q,\omega
}t\text{,}
\]
where
\[
\int_{\omega_{0}}^{x}f\left(  t\right)  d_{q,\omega}t:=\left(  x\left(
1-q\right)  -\omega\right)  \sum_{k=0}^{+\infty}q^{k}f\left(  xq^{k}
+\omega\left[  k\right]  _{q}\right)  \text{, }x\in I\, ,
\]
provided that the series converges at $x=a$ and $x=b$. In that case, $f$ is
called $q,\omega$-\emph{integrable on} $\left[a,b\right]$. We say that
$f$ is $q,\omega$-\emph{integrable over} $I$ if it is $q,\omega$-\emph{integrable}
over $[a,b]$ for all $a,b\in I$.
\end{definition}

\begin{remark}
The $q,\omega$-\emph{integral} generalizes the Jackson
$q$-integral and the N\"{o}rlund sum \cite{Kac}.
When $\omega=0$, we obtain
the Jackson $q$-integral
\[
\int_{a}^{b}f\left(  t\right)  d_{q}t:=\int_{0}^{b}f\left(  t\right)
d_{q}t-\int_{0}^{a}f\left(  t\right)  d_{q}t\text{,}
\]
where
\[
\int_{0}^{x}f\left(  t\right)  d_{q}t:=x\left(  1-q\right)  \sum_{k=0}
^{+\infty}q^{k}f\left(  xq^{k}\right)  \text{.}
\]
When $q\rightarrow1$, we obtain the N\"{o}rlund sum
\[
\int_{a}^{b}f\left(  t\right)  \Delta_{\omega}t :=\int_{+\infty}^{b}f\left(
t\right)  \Delta_{\omega}t-\int_{+\infty}^{a}f\left(  t\right)  \Delta
_{\omega}t,
\]
where
\[
\int_{+\infty}^{x}f\left(  t\right)  \Delta_{\omega}t :=-\omega\sum
_{k=0}^{+\infty}f\left(  x+k\omega\right)  \text{.}
\]
\end{remark}

It can be shown  that if $f:I\rightarrow\mathbb{R}$
is continuous at $\omega_{0}$,
then $f$ is $q,\omega$-\emph{integrable over} $I$ (see \cite{1,Aldwoah:IJMS} for the proof).

\begin{theorem}(\cite{1} Fundamental Theorem of Hahn's Calculus)
Assume that $f:I\rightarrow \mathbb{R}$ is continuous at $\omega_{0}$ and,
for each $x\in I$, define
\[
F\left(  x\right)  :=\int_{\omega_{0}}^{x}f\left(  t\right)  d_{q,\omega
}t\text{.}
\]
Then $F$ is continuous at $\omega_{0}$. Furthermore, $D_{q,\omega}\left[
F\right]  \left(  x\right)$ exists for every $x\in I$ and
$D_{q,\omega}\left[  F\right]  \left(  x\right)  =f\left(  x\right)$.
Conversely,
$\int_{a}^{b}D_{q,\omega}\left[  f\right]  \left(  t\right)  d_{q,\omega
}t=f\left(  b\right)  -f\left(  a\right)$  for all $a,b\in I$.
\end{theorem}

Aldwoah proved that the  $q,\omega$-\emph{integral} has the following properties:

\begin{theorem}(\cite{1,Aldwoah:IJMS})
\label{Propriedades do integral}
Let $f,g:I\rightarrow\mathbb{R}$ be $q,\omega$-\emph{integrable on} $I$,
$a,b,c\in I$ and $k\in\mathbb{R}$. Then,
\begin{enumerate}
\item $\displaystyle \int_{a}^{a}f\left(  t\right)  d_{q,\omega}t=0$;

\item $\displaystyle \int_{a}^{b}kf\left(  t\right)  d_{q,\omega}t=k\int_{a}^{b}f\left(
t\right)  d_{q,\omega}t$;

\item \label{eq:item3} $\displaystyle \int_{a}^{b}f\left(  t\right)  d_{q,\omega}t=-\int_{b}^{a}f\left(
t\right)  d_{q,\omega}t$;

\item $\displaystyle \int_{a}^{b}f\left(  t\right)  d_{q,\omega}t=\int_{a}^{c}f\left(
t\right)  d_{q,\omega}t+\int_{c}^{b}f\left(  t\right)  d_{q,\omega}t$;

\item $\displaystyle \int_{a}^{b}\left(  f\left(  t\right)  +g\left(  t\right)  \right)
d_{q,\omega}t=\int_{a}^{b}f\left(  t\right)  d_{q,\omega}t+\int_{a}
^{b}g\left(  t\right)  d_{q,\omega}t$;

\item Every Riemann integrable function $f$ on $I$
is $q,\omega$-\emph{integrable on} $I$;

\item \label{itm:ip} If $f,g:I\rightarrow \mathbb{R}$
are $q,\omega$-differentiable and $a,b\in I$, then
\[
\displaystyle \int_{a}^{b}f\left(  t\right)  D_{q,\omega}\left[  g\right]  \left(  t\right)
d_{q,\omega}t= \Big[f\left(  t\right)  g\left(  t\right)  \Big]_{a}^{b}
-\int_{a}^{b}D_{q,\omega}\left[  f\right]  \left(  t\right)  g\left(
qt+\omega\right)  d_{q,\omega}t.
\]
\end{enumerate}
\end{theorem}

Property~\ref{itm:ip} of Theorem~\ref{Propriedades do integral} is known as
$q,\omega$-\emph{integration by parts formula}.

\begin{lemma}(\textrm{cf.} \cite{1})
\label{positividade}
Let $b \in I$ and $f$ be $q,\omega$-\emph{integrable} over $I$.
Suppose that
$$
f(t)\geq 0, \quad \forall
t\in\left\{  q^{n}b+\omega\left[  n\right]  _{q}:n\in \mathbb{N}_{0}\right\}.
$$
\begin{enumerate}

\item If $\omega_0 \leq b$, then
$$
\int_{\omega_0}^b f(t)d_{q,\omega}t\geq 0.
$$

\item If $\omega_0 > b$, then
$$
\int_b^{\omega_0} f(t)d_{q,\omega}t\geq 0.
$$
\end{enumerate}
\end{lemma}

\begin{remark} As noted in \cite{BritoMT}
there is an inconsistency in \cite{1}.
Indeed, Lemma~6.2.7 of \cite{1} is only valid
if $b \ge \omega_0$ and $a \le b$.
\end{remark}

\begin{remark}
\label{rem:diff:int} In general, the Jackson--N\"{o}rlund integral does not satisfies the
following inequality (for a counterexample see \cite{1}):
$$
\left\vert \int_{a}^{b}f\left(  t\right)  d_{q,\omega}t\right\vert
\leq \int_{a}^{b}| f\left(  t\right)|  d_{q,\omega}t , \ \ \ a,b \in I.
$$
\end{remark}

For $s\in I$ we define
\[
\left[  s\right]  _{q,\omega}:=\left\{  q^{n}s+\omega\left[  n\right]_q  :n\in
\mathbb{N}
_{0}\right\}  \cup\left\{  \omega_{0}\right\}  \text{.}
\]

The following definition and lemma are important
for our purposes.

\begin{definition}
Let $s \in I$, $s\neq \omega_{0}$ and  $g:I\times]-\bar{\theta},\bar{\theta}[  \rightarrow \mathbb{R}$.
We say that $g\left(  t,\cdot\right)$ is differentiable at $\theta_{0}$ uniformly
in $\left[s\right]_{q,\omega}$ if for every $\varepsilon>0$ there exists $\delta>0$ such that
\[
0<\left\vert \theta-\theta_{0}\right\vert <\delta
\Rightarrow
\left\vert \frac{g\left(  t,\theta\right)  -g\left(  t,\theta_{0}\right)
}{\theta-\theta_{0}}-\partial_{2} g\left(  t,\theta_{0}\right)  \right\vert
<\varepsilon
\]
for all $t\in\left[  s\right]_{q,\omega}$,
where $\displaystyle\partial_{2}g=\frac{\partial g}{\partial\theta}$.
\end{definition}

\begin{lemma}(\cite{MalinowskaTorres})
\label{derivada do integral} Let $s \in I$, $s\neq \omega_{0}$, and
assume that $g:I\times]-\bar{\theta},\bar{\theta}[  \rightarrow
\mathbb{R}$ is differentiable at
$\theta_{0}$ uniformly in $\left[s\right]_{q,\omega}$,
$\displaystyle G\left(  \theta\right):=\int_{\omega_{0}}^{s}g\left(
t,\theta\right)  d_{q,\omega}t$ for $\theta$ near $\theta_{0}$, and
$\displaystyle\int_{\omega_{0}}^{s}\partial_{2}g\left(  t,\theta_{0}\right)
d_{q,\omega}t$ exist. Then, $G\left(  \theta\right)$
is differentiable at $\theta_{0}$ with
$G^{\prime}\left(  \theta_{0}\right)
=\displaystyle\int_{\omega_{0}}^{s}\partial_{2}g\left(  t,\theta_{0}\right)  d_{q,\omega}t$.
\end{lemma}

Let $a,b\in I$ with $a< b$. Recall that $I$ is an interval containing $\omega_0$.
We define the $q,\omega$-interval by
$$[a,b]_{q,\omega}:=\{q^na+\omega[n]_{q}:
n\in\mathbb{N}_{0}\}\cup\{q^nb+\omega[n]_{q}:
n\in\mathbb{N}_{0}\}\cup\{\omega_0\},$$
i.e., $[a,b]_{q,\omega} =\left[  a\right]_{q,\omega}\cup \left[  b\right]_{q,\omega}$.

\vspace{0.3 cm}

For $r \in\mathbb{N}$ we introduce the linear space
$\mathcal{Y}^{r} = \mathcal{Y}^{r}\left(\left[a,b\right],\mathbb{R}\right)$ by
$$
\mathcal{Y}^{r} :=
\left\{  y: I  \rightarrow \mathbb{R}\, |\,
D_{q,\omega}^{i}[y], i = 0,\ldots, r,
\text{ are bounded on $[a,b]$ and continuous at } \omega_{0}\right\}
$$
endowed with the norm
$$\left\Vert y\right\Vert_{r,\infty}:=\sum_{i=0}^{r}\left\Vert D_{q,\omega}
^{i}\left[  y\right]  \right\Vert_{\infty},$$
where
$\left\Vert y\right\Vert_{\infty}:=\sup_{t\in\left[  a,b\right]  }\left\vert
y\left(  t\right)  \right\vert$.

\begin{lemma}(\cite{MalinowskaTorres} Fundamental Lemma of the Hahn quantum variational calculus)
\label{lemma:DR} Let $f\in \mathcal{Y}^{0}.$ One has
$\int_{a}^{b}f(t)h(qt+\omega)d_{q,\omega}t=0$ for all functions
$h\in \mathcal{Y}^{0}$ with $h(a)=h(b)=0$ if and only if $f(t)=0$ for
all $t\in[a,b]_{q,\omega}$.
\end{lemma}


\section{Main results}
\label{sec:mr}

The main purpose of this paper is to generalize
the Hahn Calculus of Variations \cite{MalinowskaTorres} by considering the following $q,\omega$-variational problem
\begin{equation}
\label{P}
\mathcal{L}\left[  y\right] =  \int_{a}^{b}L\left(
t,y\left(  qt+\omega\right)  ,D_{q,\omega}\left[ y\right]
\left(  t\right)  ,y(a), y(b)
\right)  d_{q,\omega}t \longrightarrow \textrm{extr}\\
\end{equation}
where ``extr'' denotes ``extremize'' (\textrm{i.e.}, minimize or maximize).
In Subsection~\ref{subsection E-L} we obtain the Euler--Lagrange
equation for problem (\ref{P}) in the class of functions
$y \in \mathcal{Y}^{1}$ satisfying the boundary conditions
\begin{equation}
\label{boundary conditions}
y(a)=\alpha \quad \mbox{ and } \quad y(b)=\beta
\end{equation}
for some fixed $\alpha, \beta \in \mathbb{R}$. The transversality
conditions for problem (\ref{P}) are obtained in
Subsection~\ref{subsection natural boundary}.
In Subsection~\ref{isoperimetric problem} we prove necessary optimality conditions for isoperimetric problems.
A sufficient optimality condition under an appropriate convexity assumption is given in Subsection~\ref{sufficient_condition}

\vspace{0.3 cm}

\begin{definition}
A function $y\in \mathcal{Y}^{1}$ is
said to be admissible for (\ref{P})--(\ref{boundary conditions}) if it satisfies the endpoint conditions
\eqref{boundary conditions}. We say that $h\in \mathcal{Y}^{1}$ is an admissible variation for (\ref{P})--(\ref{boundary conditions})
if $h(a)=h(b)=0$.
\end{definition}

In the sequel we assume that the Lagrangian $L$ satisfies the following hypotheses:
\begin{enumerate}
\item[(H1)] $(u_0, \ldots, u_3)\rightarrow L(t,u_0, \ldots, u_3)$
is a $C^1(\mathbb{R}^{4}, \mathbb{R})$ function for any $t \in I$;

\item[(H2)] $t \rightarrow L(t, y(qt+\omega), D_{q,\omega}\left[  y\right](t),y(a), y(b))$ is continuous
at $\omega_0$ for any $y\in \mathcal{Y}^1$;

\item[(H3)]  functions $t \rightarrow \partial_{i+2}L(t, y(qt+\omega),
D_{q,\omega}\left[  y\right](t),  y(a), y(b))$,
$i=0,\cdots,3$  belong to
$\mathcal{Y}^{1}$
for all  $y \in \mathcal{Y}^1$.
\end{enumerate}

\begin{definition}
We say that $y_{\ast}$ is a local minimizer (resp. local maximizer) for problem
(\ref{P})--(\ref{boundary conditions}) if $y_{\ast}$ is an admissible function and there exists $\delta>0$ such that
\[
\mathcal{L}\left[  y_{\ast}\right]  \leq\mathcal{L}\left[  y\right]
\text{ \ \ (resp. }\mathcal{L}\left[  y_{\ast}\right]  \geq
\mathcal{L}\left[  y\right]  \text{) }
\]
for all admissible $y$ with
$\left\Vert y_{\ast}-y\right\Vert_{1,\infty}<\delta$.
\end{definition}

For fixed $y, h \in \mathcal{Y}^1$, we define the real function
$\phi$
by
$$
\phi(\varepsilon)
:=\mathcal{L}[y + \varepsilon h].
$$
The first variation for problem
\eqref{P} is defined by
$$\delta\mathcal{L}[y,h]:=\phi'(0).$$
Observe that,
\begin{equation*}
\begin{split}
\mathcal{L}&[y + \varepsilon h]=\int_a^b L(t,y(qt+\omega)+\varepsilon h(qt+\omega),D_{q,\omega}[y](t)
+ \varepsilon D_{q,\omega}[h](t), y(a) + \varepsilon h(a),\\
& y(b)+ \varepsilon h(b)) d_{q,\omega} t
=\int_{\omega_0}^b L(t,y(qt+\omega)+\varepsilon
h(qt+\omega),D_{q,\omega}[y](t)+ \varepsilon D_{q,\omega}[h](t),\\
&y(a)
+ \varepsilon h(a), y(b)+ \varepsilon h(b))
d_{q,\omega} t-\int_{\omega_0}^a L(t,y(qt+\omega)+\varepsilon
h(qt+\omega),D_{q,\omega}[y](t)\\ &+
\varepsilon D_{q,\omega}[h](t),
y(a) + \varepsilon h(a), y(b)+ \varepsilon h(b))
d_{q,\omega} t.
\end{split}
\end{equation*}
Writing
\begin{equation*}
\begin{split}
\mathcal{L}_b&[y + \varepsilon h]
=\int_{\omega_0}^b
L(t,y(qt+\omega)+\varepsilon h(qt+\omega),D_{q,\omega}[y](t)
+ \varepsilon D_{q,\omega}[h](t), y(a) + \varepsilon h(a),\\
& y(b)+ \varepsilon h(b)) d_{q,\omega} t
\end{split}
\end{equation*}
and
\begin{equation*}
\begin{split}
\mathcal{L}_a&[y + \varepsilon h]
=\int_{\omega_0}^a
L(t,y(qt+\omega)+\varepsilon h(qt+\omega),D_{q,\omega}[y](t)
+ \varepsilon D_{q,\omega}[h](t), y(a) + \varepsilon h(a),\\
& y(b)+ \varepsilon h(b)) d_{q,\omega} t,
\end{split}
\end{equation*}
we have $$\mathcal{L}[y + \varepsilon h]=\mathcal{L}_b[y +
\varepsilon h]-\mathcal{L}_a[y + \varepsilon h].$$ Therefore,
\begin{equation}\label{var}
\delta\mathcal{L}[y,h]=\delta\mathcal{L}_b[y,h]-\delta\mathcal{L}_a[y,
h].
\end{equation}

In order to simplify expressions, we  introduce the operator
$\{\cdot\}$ defined in the following way:
$$
\{y\}(t) :=(t, y(qt+\omega), D_{q,\omega}[y](t), y(a), y(b))
$$
where $y \in \mathcal{Y}^1$.

Knowing \eqref{var}, the following lemma is a direct consequence of
Lemma~\ref{derivada do integral}.

\begin{lemma}
\label{asump}
For fixed $y, h \in \mathcal{Y}^1$ let
$$g(t,\varepsilon)=L(t,y(qt+\omega)+\varepsilon
h(qt+\omega),D_{q,\omega}[y](t)+ \varepsilon D_{q,\omega}[h](t), y(a)+\varepsilon h(a), y(b)+ \varepsilon h(b))$$ for
$\varepsilon \in ]-\overline{\varepsilon}, \overline{\varepsilon}[$, for some $\overline{\varepsilon}>0$,
i.e.,
$$
g(t,\varepsilon)=L\{y+\varepsilon h\}(t).
$$
Assume
that:
\begin{itemize}
\item[(i)] $g(t,\cdot)$ is differentiable at $0$ uniformly in
$t \in [a,b]_{q,\omega}$;
\item[(ii)] $\mathcal{L}_a[y + \varepsilon h]=\displaystyle\int_{\omega_{0}}^{a}g\left(
t,\epsilon\right)  d_{q,\omega}t $ and $\mathcal{L}_b[y + \varepsilon
h]=\displaystyle\int_{\omega_{0}}^{b}g\left(
t,\epsilon\right)  d_{q,\omega}t $ exist for $\varepsilon\approx0$;
\item[(iii)] $\displaystyle \int_{\omega_0}^a\partial_2g(t,0)d_{q,\omega}t$ and
$\displaystyle \int_{\omega_0}^b\partial_2g(t,0)d_{q,\omega}t$ exist.
\end{itemize}
Then,
\begin{equation*}
\begin{split}
\delta\mathcal{L}&[y,h]
=\int_a^b \Big( \partial_2
L\{y\}(t)\cdot h(qt+\omega)
+ \partial_3 L\{y\}(t)\cdot
D_{q,\omega}[h](t) + \partial_4 L\{y\}(t)\cdot h(a) \\
&+ \partial_5 L\{y\}(t)\cdot h(b)\Big)d_{q,\omega}t.
\end{split}
\end{equation*}
\end{lemma}


\subsection{The Hahn Quantum Euler--Lagrange equation}
\label{subsection E-L}

\begin{theorem}(Necessary optimality condition
to (\ref{P})--(\ref{boundary conditions}))
\label{E-L}Under hypotheses (H1)--(H3) and conditions (i)--(iii)
of Lemma~\ref{asump} on the Lagrangian $L$, if
 $\tilde{y}$ is  a local minimizer
or local maximizer to problem
(\ref{P})--(\ref{boundary conditions}), then
 $\tilde{y}$ satisfies the Euler--Lagrange equation
\begin{equation}
\label{E-L equation}
\partial_2
L\{y\}(t) - D_{q,\omega} [\partial_3 L]\{y\}(t)=0
\end{equation}
for all $t\in\left[a,b\right]_{q,\omega}$.
\end{theorem}

\begin{proof}
Suppose that $\mathcal{L}$ has a local extremum at $\tilde{y}$. Let
$h$ be any admissible variation and define a function
$\phi : \, ]-\bar{\varepsilon},\bar{\varepsilon}[\rightarrow \mathbb{R}$
by $\phi(\varepsilon) = \mathcal{L}[\tilde{y} + \varepsilon h]$.
A necessary condition for $\tilde{y}$ to be
an extremizer is given  by $\phi^\prime(0)=0$.
Note that
\begin{equation*}
\begin{split}
\label{eq:FT} \phi^\prime(0)&=
\int_a^b \Big( \partial_2
L\{\tilde{y}\}(t)\cdot h(qt+\omega)
+ \partial_3 L\{\tilde{y}\}(t)\cdot
D_{q,\omega}[h](t) + \partial_4 L\{\tilde{y}\}(t)\cdot h(a) \\
&+ \partial_5 L\{\tilde{y}\}(t)\cdot h(b)\Big)d_{q,\omega}t.
\end{split}
\end{equation*}
Since $h(a) = h(b)= 0$, then
$$\phi^\prime(0)=
\int_a^b \Big( \partial_2
L\{\tilde{y}\}(t)\cdot h(qt+\omega)
+ \partial_3 L\{\tilde{y}\}(t)\cdot
D_{q,\omega}[h](t)\Big)d_{q,\omega}t.$$
Integration by parts
gives
\begin{equation*}
\int_a^b  \partial_3 L\{\tilde{y}\}(t)\cdot
D_{q,\omega}[h](t)d_{q,\omega}t= \Big[\partial_3 L\{\tilde{y}\}(t)\cdot h(t)\Big]^b_a - \int_a^b  D_{q,\omega} [\partial_3 L]\{\tilde{y}\}(t)\cdot
h(qt+\omega)d_{q,\omega}t
\end{equation*}
and since $h(a) = h(b)= 0$, then
\begin{equation*}
\phi^\prime(0)=0\Leftrightarrow
\int_a^b \Big( \partial_2
L\{\tilde{y}\}(t) - D_{q,\omega} [\partial_3 L]\{\tilde{y}\}(t)\Big) \cdot h(qt+\omega)d_{q,\omega}t =0.
\end{equation*}
Thus, by
Lemma~\ref{lemma:DR}, we have
\begin{equation*}
\partial_2
L\{\tilde{y}\}(t) - D_{q,\omega} [\partial_3 L]\{\tilde{y}\}(t)=0
\end{equation*}
for all $t\in[a,b]_{q,\omega}$.
\end{proof}

\begin{remark}
Under appropriate conditions, when $(\omega, q)\rightarrow (0, 1)$,
we obtain a corresponding result in the classical context of the
calculus of variations \cite{JP:DT:AZ} (see also \cite{MalTor}):
$$\frac{d}{dt}\partial_3 L(t,y(t),y^\prime(t),y(a),y(b))= \partial_2 L(t,y(t),y^\prime(t),y(a),y(b)).$$
\end{remark}

\begin{remark} In the basic problem of the calculus of variations, $L$ does not depend on $y(a)$ and $y(b)$,
and equation (\ref{E-L equation}) reduces
to the Hahn quantum Euler--Lagrange equation presented in
\cite{MalinowskaTorres}.
\end{remark}

\begin{remark}
In practical terms the hypotheses of Theorem~\ref{E-L} are not easy
to verify \emph{a priori}. However, we can assume that all
hypotheses are satisfied and apply the $q,\omega$-Euler--Lagrange
equation \eqref{E-L equation} heuristically to obtain a
\emph{candidate}. If such a candidate is, or not, a solution to the
variational problem is a different question that require further
analysis (see \S\ref{sufficient_condition} and Section~\ref{sec:Ex}).
\end{remark}


\subsection{Natural boundary conditions}
\label{subsection natural boundary}

\begin{theorem}(Natural boundary conditions to (\ref{P}))
\label{Theorem natural boundary conditions}
Under hypotheses (H1)--(H3) and conditions (i)--(iii)
of Lemma~\ref{asump} on the Lagrangian $L$, if
 $\tilde{y}$ is  a local minimizer
or local maximizer to problem
(\ref{P}),
then $\tilde{y}$ satisfies
the Euler--Lagrange equation
(\ref{E-L equation}) and
\begin{enumerate}
\item  if  $y(a)$ is free,  then the natural boundary condition
\begin{equation}
\label{a}
\partial_3 L\{\tilde{y}\}(a)= \int_a^b \partial_{4}L\{\tilde{y}\}(t)d_{q,\omega}t
\end{equation}
holds;
\item if $y(b)$ is free, then the natural boundary condition
\begin{equation}
\label{b}
\partial_3 L\{\tilde{y}\}(b)= -\int_a^b \partial_{5}L\{\tilde{y}\}(t)d_{q,\omega}t
\end{equation}
holds.
\end{enumerate}
\end{theorem}

\begin{proof}
Suppose that $\tilde{y}$ is a local minimizer (resp. maximizer) to
problem (\ref{P}). Let $h$ be any $\mathcal{Y}^1$ function. Define a
function $\phi : \,
]-\bar{\varepsilon},\bar{\varepsilon}[\rightarrow \mathbb{R}$ by
$\phi(\varepsilon) = \mathcal{L}[\tilde{y} + \varepsilon h]$. It is
clear that a necessary condition for $\tilde{y}$ to be an extremizer
is given by $\phi^{\prime}\left(  0\right)  =0$.
From the arbitrariness of $h$ and using similar arguments
as the ones used in the proof of Theorem~\ref{E-L}, it can be proved that
$\tilde{y}$ satisfies the Euler--Lagrange equation (\ref{E-L equation}).

\begin{enumerate}

\item Suppose now that $y(a)$ is free.
If $y(b)=\beta$ is given, then $h(b)=0$;
if $y(b)$ is free, then we restrict ourselves
to those $h$ for which $h(b)=0$.
Therefore,
\begin{equation}
\label{boundary condition a}
\begin{split}
0 &=  \phi^\prime(0)\\
&= \int_a^b \Big( \partial_2
L\{\tilde{y}\}(t) - D_{q,\omega} [\partial_3 L]\{\tilde{y}\}(t)\Big) \cdot h(qt+\omega)d_{q,\omega}t\\
& \quad + \displaystyle \Big( \int_a^b \partial_4 L\{\tilde{y}\}(t) d_{q,\omega}t - \partial_3 L\{\tilde{y}\}(a)\Big)\cdot h(a)=0.
\end{split}
\end{equation}
Using the Euler--Lagrange equation (\ref{E-L equation})
into (\ref{boundary condition a}) we obtain
$$
\displaystyle \Big( \int_a^b \partial_4 L\{\tilde{y}\}(t) d_{q,\omega}t - \partial_3 L\{\tilde{y}\}(a)\Big)\cdot h(a)=0.
$$
From the arbitrariness of $h$ it follows that
$$
\partial_3 L\{\tilde{y}\}(a) =\int_a^b \partial_4 L\{\tilde{y}\}(t) d_{q,\omega}t.
$$
\item Suppose now that $y(b)$ is free.
If $y(a)=\alpha$, then $h(a)=0$; if $y(a)$ is
 free, then we restrict ourselves to those $h$ for which $h(a)=0$.
Thus,
\begin{equation}
\label{boundary condition b}
\begin{split}
0 &=  \phi^\prime(0)\\
&= \int_a^b \Big( \partial_2
L\{\tilde{y}\}(t) - D_{q,\omega} [\partial_3 L]\{\tilde{y}\}(t)\Big) \cdot h(qt+\omega)d_{q,\omega}t\\
& \quad + \displaystyle \Big( \int_a^b \partial_5 L\{\tilde{y}\}(t) d_{q,\omega}t + \partial_3 L\{\tilde{y}\}(b)\Big)\cdot h(b)=0.
\end{split}
\end{equation}
Using the Euler--Lagrange equation (\ref{E-L equation})
into (\ref{boundary condition b}),
and from the arbitrariness of $h$, it follows that
$$\partial_3 L\{\tilde{y}\}(b)= -\int_a^b \partial_{5}L\{\tilde{y}\}(t)d_{q,\omega}t.$$
\end{enumerate}
\end{proof}

In the case where $L$ does not depend on $y(a)$ and $y(b)$, under
appropriate assumptions on the Lagrangian $L$ (\textrm{cf.}
\cite{MalinowskaTorres}), we obtain the following result.

\begin{corollary}\label{corollary}
If
 $\tilde{y}$ is  a local minimizer
or local maximizer to problem
\begin{equation*}
 \mathcal{L}\left[  y\right] =  \int_{a}^{b}L\left(
t,y\left(  qt+\omega\right)  ,D_{q,\omega}\left[ y\right] \left(
t\right)
\right)  d_{q,\omega}t \longrightarrow \textrm{extr}\\
\end{equation*}
then $\tilde{y}$ satisfies the Euler--Lagrange equation
\begin{equation*}
\partial_2
L\left( t,y\left(  qt+\omega\right)  ,D_{q,\omega}\left[ y\right]
\left( t\right)\right) - D_{q,\omega} [\partial_3 L]\left( t,y\left(
qt+\omega\right)  ,D_{q,\omega}\left[ y\right] \left(
t\right)\right)=0
\end{equation*}
for all $t\in\left[a,b\right]_{q,\omega}$,
 and
\begin{enumerate}

\item  if  $y(a)$ is free,  then the natural boundary condition
\begin{equation}
\label{ha}
\partial_3 L\left(
t,\widetilde{y}\left(  qa+\omega\right)  ,D_{q,\omega}\left[ \widetilde{y}\right] \left(
a\right)\right)= 0
\end{equation}
holds;
\item if $y(b)$ is free, then the natural boundary condition
\begin{equation}
\label{hb}
\partial_3 L\left(
t,\widetilde{y}\left(  bt+\omega\right)  ,D_{q,\omega}\left[ \widetilde{y}\right] \left(
b\right) \right)= 0
\end{equation}
holds.
\end{enumerate}
\end{corollary}

\begin{remark}
Under appropriate conditions, when $(\omega, q)\rightarrow (0, 1)$
equations (\ref{ha}) and (\ref{hb}) reduce to the well-known natural
boundary conditions for the basic problem of the calculus of
variations
$$
\partial_{3}L(a,\tilde{y}(a),\tilde{y}^{\prime}(a))=0
\quad \mbox{ and } \quad \partial_{3}L(b,\tilde{y}(b),\tilde{y}^{\prime}(b))=0,
$$
respectively.
\end{remark}

\subsection{Isoperimetric problem}
\label{isoperimetric problem}

We now study quantum isoperimetric problems. Both normal
and abnormal extremizers are considered. One of the earliest problem involving such a constraint is that of finding
the geometric figure with the largest area that can be enclosed by a curve of some specified length.  Isoperimetric problems
have found a broad class of important applications throughout the
centuries. Areas of application include also economy (see, \textrm{e.g.}, \cite{Almeida,Caputo} and the references given there). In the context of the quantum calculus we mention, \textrm{e.g.}, \cite{Almeida2}. The isoperimetric problem consists
of minimizing or maximizing the functional
\begin{equation}
\label{isoperimetric1}
\mathcal{L}\left[  y\right] =  \int_{a}^{b}L\left(
t,y\left(  qt+\omega\right)  ,D_{q,\omega}\left[ y\right]
\left(  t\right)  ,y(a), y(b)
\right)  d_{q,\omega}t
\end{equation}
in the class of functions $y \in \mathcal{Y}^1$ satisfying the
integral constraint
\begin{equation}
\label{integral-constraint}
\mathcal{J}\left[  y\right] =  \int_{a}^{b}F\left(
t,y\left(  qt+\omega\right)  ,D_{q,\omega}\left[ y\right]
\left(  t\right)  ,y(a), y(b)
\right)  d_{q,\omega}t =\gamma
\end{equation}
for some  $\gamma \in \mathbb{R}$.

\begin{definition}
We say that  $\tilde{y} \in \mathcal{Y}^1$ is  a local minimizer
(resp. local maximizer) for the isoperimetric problem
\eqref{isoperimetric1}--\eqref{integral-constraint} if there exists $\delta >0$ such that
$\mathcal{L}[\tilde{y}]\leq \mathcal{L}[y]$ (resp.
$\mathcal{L}[\tilde{y}] \geq \mathcal{L}[y]$) for all $y \in \mathcal{Y}^1$
satisfying  the isoperimetric
constraint \eqref{integral-constraint} and $\left\Vert \widetilde{y}-y\right\Vert_{1,\infty}<\delta$.
\end{definition}

\begin{definition}
We say that $y \in \mathcal{Y}^1$ is an extremal
to $\mathcal{J}$ if $y$ satisfies the Euler--Lagrange equation
(\ref{E-L equation}) relatively to $\mathcal{J}$.
An extremizer (\textrm{i.e.}, a local minimizer or a local maximizer) to problem
\eqref{isoperimetric1}--\eqref{integral-constraint}
that is not an extremal to $\mathcal{J}$ is said to be a normal extremizer;
otherwise, the extremizer is said to be abnormal.
\end{definition}

\begin{theorem}(Necessary optimality condition for normal extremizers
to \eqref{isoperimetric1}--\eqref{integral-constraint})
\label{normalcase}
Suppose that $L$ and $F$ satisfy hypotheses (H1)--(H3) and conditions (i)--(iii)
of Lemma~\ref{asump}, and
suppose that $\widetilde{y } \in \mathcal{Y}^1$
gives a local minimum or a local maximum
to the functional $\mathcal{L}$
subject to the integral constraint (\ref{integral-constraint}). If $\widetilde{y}$
is not an extremal to $\mathcal{J}$, then there exists
a real $\lambda$ such that $\widetilde{y}$ satisfies the equation
\begin{equation}
\label{E-L equation for H}
\partial_2
H\{y\}(t) - D_{q,\omega} [\partial_3 H]\{y\}(t)=0
\end{equation}
for all $t \in [a,b]_{q,\omega}$, where $H = L-\lambda F$  and
\begin{enumerate}
\item  if  $y(a)$ is free,  then the natural boundary condition
\begin{equation}
\partial_3 H\{\tilde{y}\}(a)= \int_a^b \partial_{4}H\{\tilde{y}\}(t)d_{q,\omega}t
\end{equation}
holds;
\item if $y(b)$ is free, then the natural boundary condition
\begin{equation}
\partial_3 H\{\tilde{y}\}(b)= -\int_a^b \partial_{5}H\{\tilde{y}\}(t)d_{q,\omega}t
\end{equation}
holds.
\end{enumerate}
\end{theorem}

\begin{proof} Suppose that $\widetilde{y} \in \mathcal{Y}^1$ is a normal
extremizer to problem \eqref{isoperimetric1}--\eqref{integral-constraint}.
Define the real functions $\phi, \psi:\mathbb{R}^2 \rightarrow \mathbb{R}$ by
\begin{gather*}
\phi(\epsilon_1, \epsilon_2)
= \mathcal{L}[\widetilde{y} + \epsilon_1 h_1 + \epsilon_2 h_2],\\
\psi(\epsilon_1, \epsilon_2)
= \mathcal{J}[\widetilde{y} + \epsilon_1 h_1 + \epsilon_2 h_2]  - \gamma,
\end{gather*}
where $h_2 \in \mathcal{Y}^1$ is  fixed  (that we will choose later) and $h_1 \in \mathcal{Y}^1$
is an arbitrary fixed function.

Note that
\begin{equation*}
\begin{split}
\frac{\partial \psi}{\partial \epsilon_2}(0,0)&= \displaystyle\int_a^b \Big( \partial_{2}F\{\widetilde{y}\}(t)\cdot h_2(qt+\omega)
+ \partial_{3}F\{\widetilde{y}\}(t) \cdot D_{q,\omega}[h_2](t) + \partial_{4}F\{\widetilde{y}\}(t)\cdot h_2(a)\\
& +
\partial_{5}F\{\widetilde{y}\}(t)\cdot h_2(b)\Big ) d_{q,\omega}t.
\end{split}
\end{equation*}
Using integration by parts formula we get

\begin{multline*}
\frac{\partial \psi}{\partial \epsilon_2}(0,0)
= \displaystyle \int_a^b \Big( \partial_{2}F\{\widetilde{y}\}(t)
- D_{q,\omega} [\partial_{3}F]\{\widetilde{y}\}(t)\Big)\cdot h_2(qt+\omega) d_{q,\omega}t\\
+\displaystyle \int_a^b \Big( \partial_{4}F\{\widetilde{y}\}(t)\cdot h_2(a) + \partial_{5}F\{\widetilde{y}\}(t)\cdot h_2(b)\Big)d_{q,\omega}t
+\Big[\partial_{3}F\{\widetilde{y}\}(t)\cdot h_2(t)\Big]_a^b.
\end{multline*}

Restricting $h_2$ to those such that $h_2(a)=h_2(b)=0$ we obtain
$$
\frac{\partial \psi}{\partial \epsilon_2}(0,0)
= \displaystyle \int_a^b \Big( \partial_{2}F\{\widetilde{y}\}(t)
- D_{q,\omega} [\partial_{3}F]\{\widetilde{y}\}(t)\Big)\cdot h_2(qt+\omega) d_{q,\omega}t.\\
$$

Since  $\widetilde{y}$ is not an extremal to $\mathcal{J}$,
then we can choose $h_2$ such that
$\displaystyle \frac{\partial \psi}{\partial \epsilon_2}(0,0) \neq 0$.
We keep $h_2$ fixed. Since $\psi(0,0)=0$, by the Implicit Function Theorem
there exists a function $g$ defined in a neighborhood $V$ of zero,
such that $g(0)=0$ and $\psi(\epsilon_1,g(\epsilon_1))=0$,
for any $\epsilon_1 \in V$, that is, there exists a subset of variation curves
$y= \widetilde{y} + \epsilon_1 h_1 + g(\epsilon_1)h_2$
satisfying the isoperimetric constraint.
Note that  $(0,0)$ is an extremizer of $\phi$ subject to the constraint $\psi=0$ and
$$
\nabla \psi (0,0) \neq (0,0).
$$

By the Lagrange multiplier rule,
there exists some constant $\lambda \in \mathbb{R}$ such that
\begin{equation}\label{gradient} \nabla \phi (0,0) = \lambda \nabla \psi (0,0).
\end{equation}

Restricting  $h_1$ to those such that $h_1(a)=h_1(b)=0$ we get
$$\frac{\partial \phi}{\partial \epsilon_1}(0,0)
= \displaystyle \int_a^b \Big( \partial_{2}L\{\widetilde{y}\}(t)
- D_{q,\omega} [\partial_{3}L]\{\widetilde{y}\}(t)\Big)\cdot h_1(qt+\omega) d_{q,\omega}t$$

and
$$\frac{\partial \psi}{\partial \epsilon_1}(0,0)
= \displaystyle \int_a^b \Big( \partial_{2}F\{\widetilde{y}\}(t)
- D_{q,\omega} [\partial_{3}F]\{\widetilde{y}\}(t)\Big)\cdot h_1(qt+\omega) d_{q,\omega}t.$$

Using (\ref{gradient}) it follows that
$$
\displaystyle \int_a^b \Bigl( \partial_{2}L\{\widetilde{y}\}(t)
- D_{q,\omega} [\partial_{3}L]\{\widetilde{y}\}(t)
 - \lambda\Big(\partial_{2}F\{\widetilde{y}\}(t)
- D_{q,\omega} [\partial_{3}F]\{\widetilde{y}\}(t)\Big)\Bigr)\cdot h_1(qt+\omega) d_{q,\omega}t=0.
$$

Using the Fundamental Lemma of the Hahn quantum variational calculus
(Lemma~\ref{lemma:DR}), and recalling that $h_1$
is arbitrary, we conclude that
$$
\partial_{2}L\{\widetilde{y}\}(t)
- D_{q,\omega} [\partial_{3}L]\{\widetilde{y}\}(t)
 - \lambda\Big(\partial_{2}F\{\widetilde{y}\}(t)
- D_{q,\omega} [\partial_{3}F]\{\widetilde{y}\}(t)\Big)=0
$$

for all $t \in [a,b]_{q,\omega}$, proving that $H= L-\lambda F$
satisfies the Euler--Lagrange condition (\ref{E-L equation for H}).

\begin{enumerate}
\item Suppose now that $y(a)$ is free.
If $y(b)=\beta$ is given, then $h_1(b)=0$;
if $y(b)$ is free, then we restrict ourselves
to those $h_1$ for which $h_1(b)=0$.
Therefore,
\begin{equation}
\begin{split}
&\frac{\partial \phi}{\partial \epsilon_1}(0,0)
= \displaystyle \int_a^b \Big( \partial_{2}L\{\widetilde{y}\}(t)
- D_{q,\omega} [\partial_{3}L]\{\widetilde{y}\}(t)\Big)\cdot h_1(qt+\omega) d_{q,\omega}t\\
& \quad + \displaystyle \Big( \int_a^b \partial_4 L\{\tilde{y}\}(t) d_{q,\omega}t - \partial_3 L\{\tilde{y}\}(a)\Big)\cdot h_1(a)
\end{split}
\end{equation}
and
\begin{equation}
\begin{split}
&\frac{\partial \psi}{\partial \epsilon_1}(0,0)
= \displaystyle \int_a^b \Big( \partial_{2}F\{\widetilde{y}\}(t)
- D_{q,\omega} [\partial_{3}F]\{\widetilde{y}\}(t)\Big)\cdot h_1(qt+\omega) d_{q,\omega}t\\
& \quad + \displaystyle \Big( \int_a^b \partial_4 F\{\tilde{y}\}(t) d_{q,\omega}t - \partial_3 F\{\tilde{y}\}(a)\Big)\cdot h_1(a).
\end{split}
\end{equation}

Using (\ref{gradient}) and
the Euler--Lagrange equation (\ref{E-L equation for H})
we obtain
$$
\displaystyle \Big( \int_a^b \partial_4 L\{\tilde{y}\}(t) d_{q,\omega}t - \partial_3 L\{\tilde{y}\}(a)\Big)\cdot h_1(a)=
\lambda \Big( \int_a^b \partial_4 F\{\tilde{y}\}(t) d_{q,\omega}t - \partial_3 F\{\tilde{y}\}(a)\Big)\cdot h_1(a).
$$
Hence
$$
\displaystyle \Big (\int_a^b \partial_4 H\{\tilde{y}\}(t) d_{q,\omega}t- \partial_3 H\{\tilde{y}\}(a)\Big)\cdot h_1(a)=0
$$
and from the arbitrariness of $h_1$ we conclude that
$$
\partial_3 H\{\tilde{y}\}(a) =\int_a^b \partial_4 H\{\tilde{y}\}(t) d_{q,\omega}t.
$$

\item Suppose now that $y(b)$ is free.
If $y(a)=\alpha$, then $h_1(a)=0$; if $y(a)$ is
 free, then we restrict ourselves to those $h_1$ for which $h_1(a)=0$.
 Using similar arguments as the ones used in (1), we obtain that
 $$\partial_3 H\{\tilde{y}\}(b)= -\int_a^b \partial_{5}H\{\tilde{y}\}(t)d_{q,\omega}t.$$
\end{enumerate}

\end{proof}

Introducing an extra multiplier $\lambda_{0}$ we can also deal with abnormal extremizers to
the isoperimetric problem \eqref{isoperimetric1}--\eqref{integral-constraint}.

\begin{theorem}(Necessary optimality condition for normal and abnormal extremizers
to \eqref{isoperimetric1}--\eqref{integral-constraint})
\label{thm:abn}
Suppose that $L$ and $F$ satisfy hypotheses (H1)--(H3) and conditions (i)--(iii)
of Lemma~\ref{asump}, and
suppose that $\widetilde{y} \in \mathcal{Y}^1$ gives a local minimum or a local maximum
to the functional $\mathcal{L}$ subject to
the integral constraint (\ref{integral-constraint}). Then there exist two constants $\lambda_0$
and  $\lambda$, not both zero, such that $\widetilde{y}$ satisfies the equation
\begin{equation}
\label{E-L equation for H normal and abnormal}
\partial_2
H\{y\}(t) - D_{q,\omega} [\partial_3 H]\{y\}(t)=0
\end{equation}
for all $t \in [a,b]_{q,\omega}$, where $H = \lambda_0 L-\lambda F$ and
\begin{enumerate}
\item  if  $y(a)$ is free,  then the natural boundary condition
\begin{equation}
\partial_3 H\{\tilde{y}\}(a)= \int_a^b \partial_{4}H\{\tilde{y}\}(t)d_{q,\omega}t
\end{equation}
holds;
\item if $y(b)$ is free, then the natural boundary condition
\begin{equation}
\partial_3 H\{\tilde{y}\}(b)= -\int_a^b \partial_{5}H\{\tilde{y}\}(t)d_{q,\omega}t
\end{equation}
holds.
\end{enumerate}
\end{theorem}

\begin{proof}  The proof is similar to the proof of Theorem \ref{normalcase}.
Since $(0,0)$ is an extremizer of $\phi$ subject
to the constraint $\psi=0$, the
abnormal Lagrange multiplier rule
(\textrm{cf.}, \textrm{e.g.}, \cite{vanBrunt})
guarantees the existence of two reals
$\lambda_0$ and $\lambda$, not both zero, such that
$$
\lambda_0 \nabla \phi = \lambda \nabla \psi.
$$
\end{proof}

\begin{remark}
Note that if $\widetilde{y}$ is a normal extremizer then,
by Theorem~\ref{normalcase}, one can choose
$\lambda_0 = 1$ in Theorem~\ref{thm:abn}.
The condition $(\lambda_0,\lambda) \ne (0,0)$ guarantees
that Theorem~\ref{thm:abn} is a useful necessary condition.
\end{remark}

In the case where $L$ and $F$ do not depend on $y(a)$ and $y(b)$, under
appropriate assumptions on Lagrangians $L$ and $F$, we obtain the following result.

\begin{corollary}\label{corollary2}
If
 $\tilde{y}$ is  a local minimizer
or local maximizer to the problem
\begin{equation*}
 \mathcal{L}\left[  y\right] =  \int_{a}^{b}L\left(
t,y\left(  qt+\omega\right)  ,D_{q,\omega}\left[ y\right] \left(
t\right)
\right)  d_{q,\omega}t \longrightarrow \textrm{extr}\\
\end{equation*}
subject to the integral constraint
$$
\mathcal{J}\left[  y\right] =  \int_{a}^{b}F\left(
t,y\left(  qt+\omega\right)  ,D_{q,\omega}\left[ y\right]
\left(  t\right)
\right)  d_{q,\omega}t =\gamma
$$
for some  $\gamma \in \mathbb{R}$,
then there exist two constants $\lambda_0$
and  $\lambda$, not both zero, such that $\widetilde{y}$ satisfies the following equation
$$
\partial_2
H(t,y(qt+\omega),D_{q,\omega}[y](t)) - D_{q,\omega} [\partial_3 H](t,y(qt+\omega),D_{q,\omega}[y](t))=0
$$
for all $t \in [a,b]_{q,\omega}$, where $H = \lambda_0 L-\lambda F$ and
\begin{enumerate}
\item if  $y(a)$ is free,  then the natural boundary condition
$$
\partial_3 H(t,\widetilde{y}(q a+\omega),D_{q,\omega}[\widetilde{y}](a))= 0
$$
holds;
\item if $y(b)$ is free, then the natural boundary condition
$$
\partial_3 H(t,\widetilde{y}(q b+\omega),D_{q,\omega}[\widetilde{y}](b))= 0
$$
holds.
\end{enumerate}
\end{corollary}


\subsection{Sufficient condition for optimality}
\label{sufficient_condition}

In this subsection we prove a sufficient optimality condition for problem (\ref{P}). Similar to the
classical calculus of variations we assume the lagrangian function to be convex (or concave).

\begin{definition}
Given a function $f:I\times \mathbb{R}^4\rightarrow \mathbb{R}$, we say that $f( t,u_1, \ldots, u_4)$
is jointly convex (resp. concave) in $(u_1, \ldots, u_4)$ if  $\partial_i
f$, $i=2,\ldots,5$, are continuous and verify the following
condition:
\begin{equation*}
f(t,u_1+\bar{u}_1,\ldots,u_4+\bar{u}_4)-f(t,u_1,\ldots,u_4) \geq
(resp. \leq) \sum_{i=2}^{5}\partial_i f(t,u_1,\ldots,u_4)\bar{u}_{i-1}
\end{equation*}
for all
$(t,u_1+\bar{u}_1,\ldots,u_4+\bar{u}_4)$,$(t,u_1,\ldots,u_4)\in
I\times\mathbb{R}^4$.
\end{definition}

\begin{theorem}
\label{suff} Let $L( t,u_1, \ldots, u_4)$ be jointly
convex (resp. concave) in $(u_1, \ldots, u_4)$. If $\tilde{y}$ satisfies
conditions (\ref{E-L equation}), (\ref{a}) and (\ref{b}), then
$\tilde{y}$ is a global minimizer (resp. maximizer) to problem (\ref{P}).
\end{theorem}

\begin{proof}
We give the proof for the convex case. Since $L$ is jointly convex
in $(u_1, \ldots, u_4)$, then  for any  $h \in \mathcal{Y}^1$,
\begin{equation*}
\begin{split}
\mathcal{L}&[\tilde{y}+h]-\mathcal{L}[\tilde{y}]=\int_a^b\left(L\{\tilde{y}+h\}(t)-L\{\tilde{y}\}(t)\right)d_{q,\omega}t\\
&\geq\int_a^b \Big( \partial_2 L\{\tilde{y}\}(t)\cdot h(qt+\omega) +
\partial_3 L\{\tilde{y}\}(t)\cdot D_{q,\omega}[h](t) + \partial_4
L\{\tilde{y}\}(t)\cdot h(a) \\
&+ \partial_5 L\{\tilde{y}\}(t)\cdot
h(b)\Big)d_{q,\omega}t.
\end{split}
\end{equation*}
Proceeding analogously as in the proof of Theorem~\ref{E-L} and since
$\tilde{y}$ satisfies conditions (\ref{E-L equation}), (\ref{a}) and
(\ref{b}), we obtain $\mathcal{L}(\tilde{y}+h) -
\mathcal{L}(\tilde{y}) \geq 0$, proving the desired result.
\end{proof}


\section{Illustrative examples and applications} \label{sec:Ex}

We provide some examples in order to illustrate our main results.

\begin{example}
\label{example1} Let $q\in ]0,1[$ and  $\omega\geq 0$ be fixed real numbers, and $I$ be
an interval of $\mathbb{R}$ such that $\omega_0,0,1\in I$.
Consider the problem
\begin{equation}
\label{ex1} \mathcal{L}[y] =\int_{0}^{1}\left(y(qt+\omega)
+\frac{1}{2} (D_{q,\omega}[y](t))^2\right)d_{q,\omega} t
\longrightarrow \textrm{min}
\end{equation}
over all $y \in \mathcal{Y}^1$ satisfying the boundary condition $y(1)=1$.
If $\widetilde{y}$ is a local
minimizer to problem \eqref{ex1}, then by Corollary~\ref{corollary} it satisfies the following conditions:
\begin{equation}
\label{EL:ex1} D_{q,\omega}[D_{q,\omega}[\widetilde{y}]](t)=1,
\end{equation}
for all $t\in\{\omega[n]_q:
n\in\mathbb{N}_{0}\}\cup\{q^n+\omega[n]_q: n\in\mathbb{N}_{0}\} \cup
\{\omega_0\}$  and
\begin{equation}
\label{a:ex1}
D_{q,\omega}[\widetilde{y}](0)=0.
\end{equation}
It is easy to verify  that
$\widetilde{y}(t)=\frac{1}{q+1}t^2-(\frac{\omega}{q+1}-c)t+d$, where $c$,
$d\in\mathbb{R}$, is a solution to equation \eqref{EL:ex1}.
Using the natural boundary condition
\eqref{a:ex1} we obtain that $c=0$.
In order
to determine $d$ we use the fixed boundary condition $y(1)=1$, and obtain that $d=\frac{q+\omega}{q+1}$.
Hence
\begin{equation*}
\widetilde{y}(t)=\frac{1}{q+1}t^2-\frac{\omega}{q+1}t
+\frac{q+\omega}{q+1}
\end{equation*}
is a candidate to be a minimizer to problem \eqref{ex1}. Moreover, since $L$ is jointly convex, by
Theorem~\ref{suff}, $\widetilde{y}$ is a global minimizer to problem \eqref{ex1}.
\end{example}

\begin{example}
\label{example2} Let $q\in ]0,1[$ and  $\omega\geq 0$ be fixed real numbers, and $I$ be
an interval of $\mathbb{R}$ such that $\omega_0,0,1\in I$.
Consider the problem
\begin{equation}
\label{AD} \mathcal{L}[y] =\int_{0}^{1}\left(y(qt+\omega)
+\frac{1}{2} (D_{q,\omega}[y](t))^2+\gamma\frac{1}{2}
(y(1)-1)^2+\nu\frac{1}{2} y^2(0)\right)d_{q,\omega} t
\longrightarrow \textrm{min}
\end{equation}
where $\gamma$, $\nu\in\mathbb{R}^+$. If $\widetilde{y}$ is a local
minimizer to problem \eqref{AD}, then by Theorem~\ref{Theorem
natural boundary conditions} it satisfies the following conditions:
\begin{equation}
\label{EL:ex} D_{q,\omega}[D_{q,\omega}[\widetilde{y}]](t)=1,
\end{equation}
for all $t\in\{\omega[n]_q:
n\in\mathbb{N}_{0}\}\cup\{q^n+\omega[n]_q: n\in\mathbb{N}_{0}\} \cup
\{\omega_0\}$, and
\begin{equation}
\label{a:ex}
D_{q,\omega}[\widetilde{y}](0)=\int_{0}^{1}\nu
\widetilde{y}(0)d_{q,\omega} t,
\end{equation}
\begin{equation}
\label{b:ex}D_{q,\omega}[\widetilde{y}](1)=-\int_{0}^{1}\gamma
(\widetilde{y}(1)-1)d_{q,\omega} t.
\end{equation}
As in Example~\ref{example1},
$\widetilde{y}(t)=\frac{1}{q+1}t^2-(\frac{\omega}{q+1}-c)t+d$, where $c$,
$d\in\mathbb{R}$, is a solution to equation \eqref{EL:ex}. In order
to determine $c$ and $d$ we use the natural boundary conditions
\eqref{a:ex} and \eqref{b:ex}. This gives
\begin{equation}\label{sol}
\widetilde{y}(t)=\frac{1}{q+1}t^2-\frac{\omega(\nu+\gamma)-\nu(\gamma-1)(q+1)+\gamma\nu}{(q+1)(\gamma+\nu\gamma+\nu)}t
+\frac{(\gamma-1)(q+1)-\gamma(1-\omega)}{(q+1)(\gamma+\nu\gamma+\nu)}
\end{equation}
as a candidate to be a minimizer to problem \eqref{AD}. Moreover, since $L$ is jointly convex, by
Theorem~\ref{suff} it is a global minimizer. The minimizer \eqref{sol} is represented in Figure~\ref{fig1} for fixed $\gamma=\nu=2$, $q=0.99$ and different values of $\omega$.

\begin{figure}
\begin{center}
\includegraphics[scale=0.40]{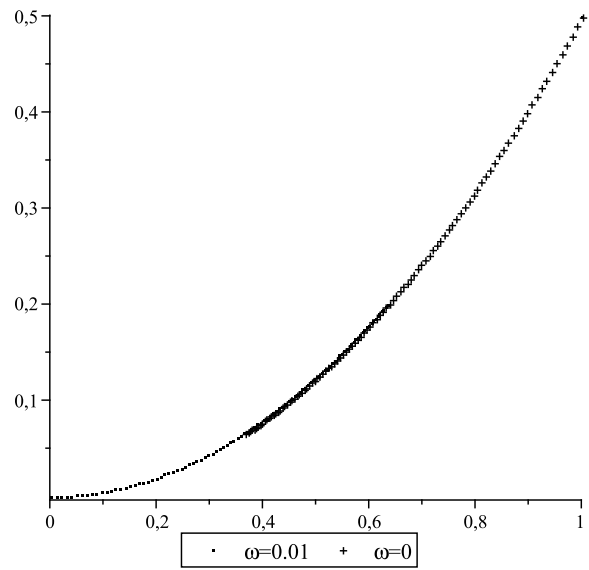}
\caption{The minimizer \eqref{sol} of Example~\ref{example2} for fixed $\gamma=\nu=2$, $q=0.99$ and different values of $\omega$.}
\label{fig1}
\end{center}
\end{figure}

We note that in the limit, when $\gamma,\nu \rightarrow +\infty$,
$\widetilde{y}(t)=\frac{1}{q+1}t^2+\frac{q}{q+1}t$ and coincides with the
solution of the following problem with fixed initial and terminal
points (\textrm{cf.} \cite{MalinowskaTorres}):
\begin{equation*}
\quad \mathcal{L}[y] =\int_{0}^{1}\left(y(qt+\omega) +\frac{1}{2}
(D_{q,\omega}[y](t))^2\right)d_{q,\omega} t \longrightarrow
\textrm{min}
\end{equation*}
subject to the boundary conditions
\begin{equation*}
\label{eq:bc} y(0)=0,  \quad y(1)=1.
\end{equation*}
Expression $\gamma\frac{1}{2} (y(1)-1)^2+\nu\frac{1}{2} y^2(0)$
added to the Lagrangian $y(qt+\omega) +\frac{1}{2}
(D_{q,\omega}[y](t))^2$ works like a penalty function when $\gamma$
and $\nu$ go to infinity. The penalty function itself grows, and
forces the merit function \eqref{AD} to increase in value when the
constraints $y(0) = 0$ and $y(1) = 1$ are violated, and causes no
growth when constraints are fulfilled.
The minimizer \eqref{sol} is represented in Figure~\ref{fig2} for fixed $q=0.5$, $\omega=1$ and different values of $\gamma$ and $\nu$.

\begin{figure}
\begin{center}
\includegraphics[scale=0.40]{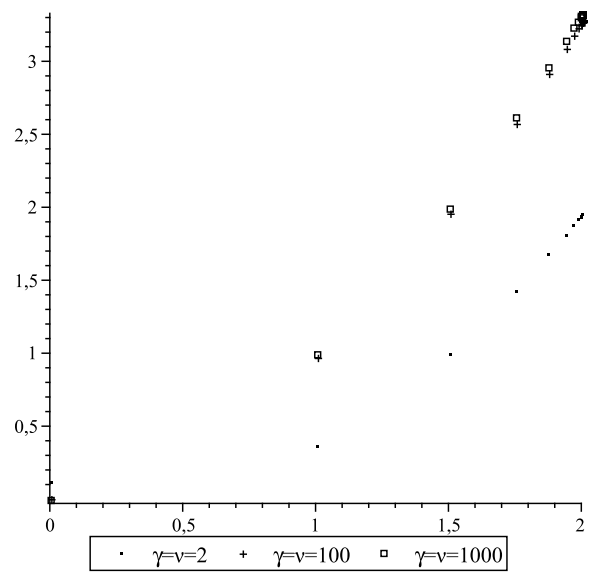}
\caption{The minimizer \eqref{sol} of Example~\ref{example2} for fixed $q=0.5$, $\omega=1$ and different values of $\gamma$ and $\nu$.}
\label{fig2}
\end{center}
\end{figure}
\end{example}

\begin{remark}
Let $$\mathcal{L}[y] =\int_{0}^{1}\left(y(qt+\omega) +\frac{1}{2}
(D_{q,\omega}[y](t))^2\right)d_{q,\omega} t$$
and
$$\widetilde{y}_1(t)=\frac{1}{q+1}t^2-\frac{\omega}{q+1}t
+\frac{q+\omega}{q+1} \quad \mbox{and} \quad \widetilde{y}_2(t)=\frac{1}{q+1}t^2+\frac{q}{q+1}t.$$
Comparing Example \ref{example1} and Example \ref{example2}, we can conclude that
$$\mathcal{L}[\widetilde{y}_1] < \mathcal{L}[\widetilde{y}_2].$$
\end{remark}

In the next example we analyze an adjustment model in economics. For
a deeper discussion of this model we refer the reader to \cite{Sen}.

\begin{example}\label{ex:ad}
Consider the dynamic model of adjustment
\begin{equation*}
\mathcal{J}[y] =\sum_{t=1}^Tr^t\left[\alpha (y(t)-\bar{y}(t))^2
+(y(t)-y(t-1))^2)\right] \longrightarrow \textrm{min},
\end{equation*}
where $y(t)$ is the output (state) variable, $r > 1$ is the
exogenous rate of discount and $\bar{y}(t)$ is the desired target
level, and $T$ is the horizon. The first component of the loss
function above is the disequilibrium cost due to deviations from
desired target and the second component characterizes the agent's
aversion to output fluctuations. In the continuous case the
objective function has the form
\begin{equation*}
\mathcal{J}[y] =\int_{1}^Te^{(r-1)t}\left[\alpha
(y(t)-\bar{y}(t))^2 +(y'(t))^2)\right] \longrightarrow \textrm{min}.
\end{equation*}
Let $q\in ]0,1[$ and  $\omega\geq 0$ be fixed real numbers, and $I$ be an
interval of $\mathbb{R}$ such that $\omega_0,0,T\in I$. The quantum
model in terms of the Hahn operators which we wish to minimize is
\begin{equation}
\label{model} \mathcal{J}[y] =\int_{0}^{T}E(1-r,t)\left[\alpha
(y(qt+\omega)-\bar{y}(qt+\omega))^2 +
(D_{q,\omega}[y](t))^2\right]d_{q,\omega} t \longrightarrow
\textrm{min},
\end{equation}
where $E\left(z,\cdot\right)$ is the $q,\omega$-exponential function
defined by
\begin{equation*}
E\left(z,t\right):=\prod _{k=0}^{\infty}(1+zq^k(t(1-q)-\omega))
\end{equation*}
for $z\in \mathbb{C}$. Several nice properties of the
$q,\omega$-exponential function can be found in
\cite{1,Aldwoah:IJMS}. By Theorem~\ref{Theorem natural boundary
conditions}, a solution to problem \eqref{model} should satisfy the
following conditions
\begin{equation}
\label{EL:ap} E(1-r,t)\left[\alpha
(y(qt+\omega)-\bar{y}(qt+\omega))\right]=D_{q,\omega}\left[E(1-r,\cdot)D_{q,\omega}[y]\right](t),
\end{equation}
for all $t\in\{\omega[n]_q:
n\in\mathbb{N}_{0}\}\cup\{Tq^n+\omega[n]_q: n\in\mathbb{N}_{0}\}
\cup \{\omega_0\}$; and
\begin{equation}
\label{a:ap} \left.E(1-r,t)D_{q,\omega}[y](t)\right|_{t=0}=0,\quad
\left.E(1-r,t)D_{q,\omega}[y](t)\right|_{t=T}=0.
\end{equation}
Taking the $q,\omega$-derivative of the right side of \eqref{EL:ap}
and applying properties of the $q,\omega$-exponential function, for
$t$ such that $|t-\omega_0|<\frac{1}{(r-1)(1-q)}$, we can rewrite
\eqref{EL:ap} and \eqref{a:ap} as
\begin{equation}
\label{El:ap:1} \left[1-(r-1)(t(1-q)-\omega)\right]\alpha
(y(qt+\omega)-\bar{y}(qt+\omega))=(r-1)D_{q,\omega}[y](t)+D_{q,\omega}[D_{q,\omega}[y]](t),
\end{equation}
\begin{equation}
\label{a:ap:1} \left.D_{q,\omega}[y](t)\right|_{t=0}=0,\quad
\left.D_{q,\omega}[y](t)\right|_{t=T}=0.
\end{equation}
Note that for $(q,\omega )\rightarrow (1,0)$ equations
\eqref{El:ap:1} and \eqref{a:ap:1} reduce to
\begin{equation*}
\alpha(y(t)-\bar{y}(t))=(r-1)y'(t)+y''(t),
\end{equation*}
\begin{equation*}
 \left.y'(t)\right|_{t=0}=0,\quad \left.y'(t)\right|_{t=T}=0,
\end{equation*}
which are necessary optimality conditions for the continuous model.

\end{example}


\section{Conclusions} \label{sec:con}

In this paper we prove optimality conditions for quantum variational problems with a
Lagrangian depending on the unspecified end-points $y(a)$, $y(b)$. Our approach uses
the quantum derivative in the forward sense:
\[
D_{q,\omega}\left[  f\right]  \left(  t\right)  :=\frac{f\left(  qt+\omega
\right)  -f\left(  t\right)  }{\left(  q-1\right)  t+\omega},  \quad t\neq\omega_0
\]
where $q\in ]0,1[$ and $\omega\geq0$, which corresponds to the delta approach in the time scale context.
However, sometimes with respect to applications (see \cite{A:U:08,Atici,M:T,MalinowskaTorres2})
the backward approach is preferable. In this sense the quantum operator
\[
D_{q,\omega}\left[  f\right]  \left(  t\right)  :=\frac{f\left(  qt+\omega
\right)  -f\left(  t\right)  }{\left(  q-1\right)  t+\omega},  \quad t\neq\omega_0
\]
where $q\in ]0,1[$ and $\omega<0$, could be considered.
Other interesting open question consists of finding a solution of equation \eqref{El:ap:1}.
As we have observed choosing particular values of $r$, $\alpha$, and a target function,
a numerical method should be used in order to solve the Euler-Lagrange equation for the problem in Example~\ref{ex:ad}.
Those issues need to be examined further and will be considered in the future.


\section*{Acknowledgments}

The authors are grateful to the support of the \emph{Portuguese
Foundation for Science and Technology} (FCT) through the
\emph{Center for Research and Development in Mathematics and
Applications} (CIDMA). Agnieszka B. Malinowska is also supported by BUT Grant S/WI/2/2011.
We would like to sincerely thank the reviewer for her/his constructive comments.



\end{document}